\documentclass[11pt,reqno]{amsart}
\usepackage{pxfonts}
\usepackage{amsmath}
\usepackage{amssymb}
\usepackage{amsthm}

\newcommand{\cV}{\mathcal{V}}

\newcommand{\cW}{\mathcal{W}}

\theoremstyle{plain}
\newtheorem{theorem}{Theorem}[section]
\newtheorem{lemma}[theorem]{Lemma}
\newtheorem{prop}[theorem]{Proposition}
\newtheorem{coro}[theorem]{Corollary}

\theoremstyle{definition}
\newtheorem{rema}[theorem]{Remark}
\newtheorem{defi}[theorem]{Definition}

\numberwithin{equation}{section}

\begin{document}

\title[Inverses of structured vector bundles]{Inverses of structured vector bundles}

\author[I. Biswas]{Indranil Biswas}
\address{School of Mathematics,
Tata Institute of fundamental research, Homi Bhabha road, Mumbai 400005, India}
\email{indranil@math.tifr.res.in}

\author[V. P. Pingali]{Vamsi P. Pingali}
\address{Department of Mathematics,
412 Krieger Hall, Johns Hopkins University, Baltimore, MD 21218, USA}
\email{vpingali@math.jhu.edu}



\begin{abstract}
Structured vector bundles were introduced by J. Simons and D. Sullivan in \cite{SS}.
We prove that all structured vector bundles whose holonomies
lie in ${\rm GL}(N,\mathbb{C})$,\, ${\rm SO}(N,\mathbb{C})$, or ${\rm Sp}(2N,
\mathbb{C})$ have structured inverses. This generalizes a theorem of Simons and
Sullivan proved in \cite{SS}.
\end{abstract}

\maketitle

\section{Introduction}

Differential K-theory is an enhanced version of topological K-theory constructed
by incorporating connections and differential forms. It was developed in order to
refine the families of Atiyah-Singer
index theorem and also to classify the Ramond-Ramond field strengths in string theory
\cite{bun}, \cite{free}. One model of the first group in differential K-theory
$\widehat{K}^0(X)$ for a manifold $X$ is a Grothendieck group of vector
bundles equipped with connections and odd differential forms \cite{kar}. In \cite{SS}, Simons
and Sullivan constructed another model of $\widehat{K}^0(X)$ using just vector
bundles and connections. They got rid of the differential form at the cost of
introducing an equivalence relation among the connections on the vector bundles.
More precisely,
they constructed $\widehat{K}^0(X)$ as the Grothendieck group of ``structured''
vector bundles; the definition of a structured vector bundle is recalled in
Section \ref{sec2}.

In addition to constructing a model of $\widehat{K}^0(X)$, Simons
and Sullivan in \cite{SS} proved an interesting result about
the existence of stable inverses of hermitian structured bundles. (It is Theorem 1.15 in
\cite{SS} which is essentially Theorem \ref{main} below for the unitary group.) However,
their proof works only for the unitary group (and also to the more general
case of compact Lie
groups) because they used the existence of universal connections \emph{\'a la}
Narasimhan-Ramanan \cite{NR}. In \cite{PT} a slightly different proof of the theorem
of Simons and Sullivan was given that did not involve universal connections. In fact,
the proof in \cite{PT} is valid even for connections that are not compatible with
the metric.

All the flat connections considered here will have trivial monodromy representation.
Note that if $\nabla$ is a flat connection with trivial monodromy on a vector bundle
$V$ over $X$, and $x_0$ is a point of $X$, then there is a unique isomorphism $f$ of
$V$ with the trivial vector bundle $X\times V_{x_0}\,\longrightarrow\, X$, equipped with
the trivial connection, such that $f$ is connection preserving and coincides with
the identity map over $x_0$ (here $V_{x_0}$ denotes the fiber of $V$ over $x_0$).

In this paper we prove the following theorem.

\begin{theorem}
Let $G$ be one of the groups ${\rm GL}(N,\mathbb{C})$,~
${\rm SO}(N,\mathbb{C})$,~ ${\rm Sp}(2N, \mathbb{C})$. Given a structured vector bundle
$\mathcal{V} \,=\, [V\, , \{ \nabla \} ]$ on a smooth manifold $X$ such that the holonomy
of some equivalent connection $\nabla$ is in $G$, there exists a structured inverse
$\mathcal{W} \,=\, [W\, , \{ \widetilde{\nabla} \} ]$ with the property that
the holonomy of $\widetilde{\nabla}$ is in the same $G$, satisfying
$$\mathcal{V} \oplus \mathcal {W} \,=\, [X \times \mathbb{C}^M \, ,
\{ \nabla _F \}]$$
where $\nabla_F$ is a flat connection with trivial monodromy on the trivial
vector bundle $X \times \mathbb{C}^M$ over $X$.
\label{main}
\end{theorem}

As an immediate consequence of Theorem \ref{main} we have the following corollary.

\begin{coro}
Let $\widehat{K^0}_G(X)$ be the Grothendieck group of structured vector bundles
$\mathcal{V} \,=\, [V,\{ \nabla \}]$ satisfying the
condition that the connection $\nabla$ has holonomy in $G$
(both $G$ and $X$ are as in Theorem \ref{main}). Let $d$ denote the trivial flat connection on a trivial bundle $X\times \mathbb{C}^k$. Then the following two hold:
\begin{enumerate}
\item Every element of $\widehat{K^0}_G(X)$ is of the form $\cV - [k]$ where $[k]
\,=\, [X \times \mathbb{C}^k, \{d\}]$ (so $k\,=\, N$ if $G\,=\, {\rm GL}(N,\mathbb{C})$
or ${\rm SO}(N,\mathbb{C})$ and $k\,=\, 2N$ if $G\,=\, {\rm Sp}(2N, \mathbb{C})$), and
\item $\cV \,= \,\cW$ in $\widehat{K^0}_G(X)$ if and only if $\cV \oplus [\mathcal{N}]
\,=\, \cW \oplus [\mathcal{N}]$ as structured bundles for some flat bundle
$[\mathcal{N}]$ with trivial monodromy.
\end{enumerate}
\label{group}
\end{coro}

In \cite{SS}, Simons and Sullivan proved what they called the Venice lemma which 
essentially says that every exact form arises out of Chern character forms of trivial 
bundles (see also \cite{PT}). Using the same ideas as in the proof of Theorem 
\ref{main} we give a proof of the following holonomy version of the Venice lemma.

\begin{prop}\label{holovenice}
Fix $G$ to be one the groups ${\rm GL}(N,\mathbb{C})$,
${\rm SO}(N,\mathbb{C})$ and ${\rm Sp}(2N, \mathbb{C})$.
If $\eta$ is any odd smooth form on $X$, then there exists a trivial bundle $T\,=\,
X \times \mathbb{C}^k$ ($k$ is as in Corollary \ref{group})
and a connection $\nabla$ on it whose holonomy is in $G$ such that
\begin{equation}\label{veniceeq}
\mathrm{ch}(T, \nabla) - \mathrm{ch}(T, d) \,=\, d\eta\, .
\end{equation}
\end{prop}

\section{Preliminaries}\label{sec2}

As mentioned in the introduction, in order to define structured vector bundles we
need to define an equivalence relation between connections on vector bundles
on a smooth manifold. To do so, we recall the
definition of the Chern-Simons forms. Throughout, $V$ and $W$ are smooth complex
vector bundles on a smooth manifold $X$. For a connection $\nabla$ on a vector
bundle $V$, let $F_{\nabla}\,\in\, C^\infty(X,\, End(V)\otimes \bigwedge^2 T^*X)$
be the curvature of $\nabla$, and let $$\mathrm{ch} (\nabla) \,=\,
\mathrm{Tr} \exp \left(\frac{\sqrt{-1}}{2\pi} F_{\nabla}\right)$$ be the corresponding
Chern character form on $X$.

\begin{defi}
If $\nabla _1$ and $\nabla _2$ are smooth connections on $V$, then the Chern-Simons form
between them is defined as a sum of odd differential forms $\mathrm{CS}(\nabla _1, \nabla _2)$
modulo exact forms satisfying the following two conditions:
\begin{enumerate}
\item (\emph{Transgression})~\ ~$d\mathrm{CS}(\nabla _1\, , \nabla _2) \,=\,
\mathrm{ch}(\nabla _1) - \mathrm{ch}(\nabla_2)$.

\item (\emph{Functoriality})~\ ~ If $f \,:\, Y \,\longrightarrow\, X$ is a smooth map between
smooth manifolds $Y$ and $X$, then $\mathrm{CS}(f^{*} \nabla _1\, , f^{*} \nabla _2) \,=\,
f^{*} \mathrm{CS}(\nabla _1\, , \nabla _2)$ modulo exact forms.
\end{enumerate}
\label{cs}
\end{defi}

In \cite{SS} an equivalence relation between connections was defined, which we now
recall.

\begin{defi}
If $\nabla _1$ and $\nabla_2$ are two smooth connections on a vector bundle
$V$ on $X$, then $$\nabla _1 \,\thicksim\, \nabla _2$$ if
$\mathrm{CS} (\nabla _1, \nabla _2)\,=\, 0$ modulo exact forms
on $X$. The equivalence class of $\nabla$ is denoted by $\{\nabla \}$.
\label{eq}
\end{defi}

An isomorphism class $\mathcal{V} \,=\, [V\, , \{ \nabla \}]$ as in Definition
\ref{eq} is called a
\emph{structured vector bundle}. The direct sum of $\mathcal{V}\,=\,[V\, , \{ \nabla
_V \}]$ and $\mathcal{W} \,=\, [W\, , \{ \nabla_W \}]$ is defined as
$$
\mathcal{V}\oplus\mathcal{W}\,=\, [V\oplus W\, , \{ \nabla_V \oplus \nabla _W \}]\, .
$$

A \textit{symplectic} (respectively, \textit{orthogonal}) bundle on $X$ is a pair
$(E \, ,\varphi)$, where $E$ is a $C^\infty$ vector bundle on $X$ and $\varphi$
is a smooth section of $E^* \otimes E^*$, such that
\begin{enumerate}
\item{} The bilinear form on $E$ defined by $\varphi$ is anti-symmetric (respectively, symmetric), and

\item{} the homomorphism
\begin{equation}\label{e1}
E\,\longrightarrow\, E^*
\end{equation}
defined by contraction of $\varphi$ is an isomorphism, equivalently, the form
$\varphi$ is fiber-wise non-degenerate.
\end{enumerate}

A connection on a vector bundle $E$ induces a connection on $E^* \otimes E^*$. A
\textit{symplectic} (respectively, \textit{orthogonal}) connection on a symplectic
(respectively, orthogonal) bundle $(E \, ,\varphi)$ is a $C^\infty$ connection
$\nabla$ on the vector bundle $E$ such that the section $\varphi$ is parallel with
respect to the connection on $E^* \otimes E^*$ induced by $\nabla$.

If $(E \, ,\varphi)$ is a symplectic (respectively, orthogonal) bundle, then the
inverse of the isomorphism in \eqref{e1} produces a symplectic (respectively,
orthogonal) structure $\varphi'$ on $E^*$, because $(E^*)^*\,=\, E$. Let
$\nabla$ be a symplectic (respectively, orthogonal) connection on the
symplectic (respectively, orthogonal) bundle $(E \, ,\varphi)$. Then the
connection $\nabla'$ on $E^*$ induced by $\nabla$ is a symplectic
(respectively, orthogonal) connection on $(E^*, \varphi')$, where $\varphi'$
is defined above. We note that the isomorphism in \eqref{e1} takes $\varphi$ and
$\nabla$ to $\varphi'$ and $\nabla'$ respectively.

We define the holonomy version of differential K-theory next.
\begin{defi}
Let $V$ be a vector bundle with a connection $\nabla _V$ whose holonomy
is in a group $G$. Let $\{ \nabla _V \}_G$ denote the
equivalence class of all such connections and denote the corresponding
structured bundles by $\cV$. Also, let $T_G$ denote the free group of such structured
bundles. The following group is the holonomy version of differential K-theory on a smooth
manifold $X$:
\begin{gather}
\widehat{K^0}_G (X) = \frac{T_G}{\cV+\cW - \cV \oplus \cW}
\end{gather}
\label{holodiffK}
\end{defi}

\begin{rema}
Notice that we require all the equivalent connections in $\{ \nabla_V \}_G$ to have holonomy in $G$ as a part of the definition
of equivalence. In particular, equivalent connections in the sense of \cite{SS} do not necessarily have their holonomy in the same group.
\end{rema}

From now onwards we drop the subscript $G$ whenever it is clear from the context.

\section{Proof of Theorem \ref{main}}

We divide the proof of Theorem \ref{main} into two cases.

\subsection{The case of $\mathbf{G=GL(N,\mathbb{C})}$}\label{s2.1}
\textbf{}\\
 This case has already been covered in \cite{PT}. However here we provide a different proof.
Our approach relies on Lemma \ref{trivsublemma} proved below. We believe that
Lemma \ref{trivsublemma} maybe of interest in its own right. The geometrical content
of the above mentioned lemma is that on $\mathbb{R}^n$ every trivial bundle with
a connection is a subbundle, equipped with the induced connection, of a trivial bundle
equipped with a flat connection.

\begin{lemma}
Let $V$ be a trivial complex vector bundle of rank $r$ on $\mathbb{R}^n$, and let
$A$ be a connection on $V$. Then there
exists an invertible, smooth $(2n+2)r \times (2n+2)r$ complex matrix valued function $g$
such that $A_{ij}\,=\,[dg g^{-1} ] _{ij}$, where $1\,\leq\, i\, , j \,\leq\, r$.
\label{trivsublemma}
\end{lemma}

\begin{proof}
Notice that $A \,=\, \displaystyle \sum _{k=1} ^n A_k dx^k$, where $A_k$ are smooth
$r\times r$ complex matrix valued functions and $x^k$ are coordinates on $\mathbb{R}^n$.
We may write $A_k$ as
$$
A_k \,=\, 2I+A_k ^{\dag} A_k+A_k - (2I+A_k ^{\dag} A_k)\, .
$$
Using this it can be deduced that $A_k$ is a difference of two smooth functions with
values in $r\times r$ positive definite
matrices (an $r\times r$ matrix $B$ is called positive definite if
$v^{\dag} (B+B^{\dag}) v \,>\, 0 \ \forall \ v \neq 0$). Indeed, we have
$$
I+A_k ^{\dag} A_k+\frac{A_k + A_k ^{\dag}}{2} \,=\, \left(I+\frac{A_k}{2}\right)^{\dag}
\left(I+\frac{A_k}{2}\right) +\frac{3}{4} A_k ^{\dag} A_k \,\geq\, 0\, .
$$ Also, $dx^k \,=\, e^{-x^k} d(e^{x^k})$ and
$-dx^k\,=\, e^{x^k} d(e^{-x^k})$. Hence
$$A\,=\, \displaystyle \sum _{k=1} ^{2n} f_k dh_k$$ where the $h_i$ are
positive smooth functions and the $f_i$ are $r\times r$ positive-definite smooth matrix-valued functions.

We may attempt to find $g$ by forcing the first $r\times (2n+2)r$ sub-matrix of $dg$ to be
$$\left [ \begin{array}{ccccc}
dh_1 Id_{r\times r} &\ldots & dh_{2n} Id_{r\times r} & 0 & 0
\end{array} \right ]$$
and the first $(2n+2)r\times r$ sub-matrix of $g^{-1}$ to be $$\left [ \begin{array}{c} f_1 \\ f_2 \\ \vdots \\ f_{2n} \\ -\sum h_k f_k \\ Id_{r\times r} \end{array} \right ].$$ If we manage to find such a $g$, then $A_{ij} \,=\, [dg g^{-1}]_{ij}$. \\

Indeed, we claim that the matrix $g$ defined by
\begin{gather}
 g=\left [ \begin{array}{cccccc}
h_1 Id_{r\times r} & h_2 Id_{r\times r} & \ldots & h_{2n}Id_{r\times r} & Id_{r\times r} & Id_{r\times r} \\
Id_{r\times r} & 0 & \ldots & 0 & f_1 \left(\displaystyle \sum_k h_k f_k\right) ^{-1} & 0 \\
0 & Id_{r\times r} & \ldots & 0 & f_2\left(\displaystyle \sum_k h_k f_k \right)^{-1} & 0 \\
\vdots & \vdots & \vdots & \vdots & \vdots & \vdots \\
0 & 0 & \ldots & Id_{r\times r} & f_{2n}\left(\displaystyle \sum_k h_k f_k\right)^{-1} & 0 \\
0 & 0 & \ldots & 0 & \left(\displaystyle \sum _k h_k f_k\right)^{-1} & Id_{r\times r} \\
\end{array} \right ]\nonumber
\end{gather}
does the job. (Here $Id_{r\times r}$ is the $r\times r$ identity matrix.) This can be verified by a
straightforward computation. Note that $\displaystyle \sum _k h_k f_k$ is invertible because $h_k>0$ and $f_k + f_k ^{\dag} >0$. 
\end{proof}

Since $\mathbb{R}^n$ is simply connected, any flat bundle on it has trivial monodromy.
Lemma \ref{trivsublemma} implies the following generalization:

\begin{prop}
Let $(V\, ,\nabla \,=\, d+A)$ be a complex rank $r$ vector bundle equipped with a
connection on a smooth manifold $X$ of
dimension $n$. Then there exists a trivial
complex vector bundle $T$ of rank $(4n+8r+2)(n+2r)$ on $X$, and a
smooth flat connection with trivial monodromy
$\widetilde{\nabla}\,=\,d+\widetilde{A}$ on $T$, such that
\begin{itemize}
\item $V\oplus W \,=\, T$ for some
smooth complex vector bundle $W$, and

\item{} the connection $A$ is induced from $\widetilde{A}$.
\end{itemize}
\label{trivsub}
\end{prop}

\begin{proof}
Using the Whitney embedding theorem, there is an embedding of the total space of $V$ in $\mathbb{R}^{2n+4r}$. The zero
section of $V$ is diffeomorphic to $X$ (and hence $X$ also sits in $\mathbb{R}^{2n+4r}$). The tangent
bundle $TV$ of $V$ is a subbundle of $T\mathbb{R}^{2n+4r}\vert_V$.

By endowing $V$ with the metric induced from the Euclidean metric on $\mathbb{R}^{2n+4r}$, we
may find the orthogonal
complement $TV^{\perp}$ of $TV$.
It satisfies $$TV\vert_X\oplus TV^{\perp} \,=\, T\mathbb{R}^{2(n+2r)}\vert_X\, .$$ The vector
bundle $V$ itself may be identified with a subbundle
of $TV$. Using the induced metric we may find the orthogonal complement $V^{\perp}$ of $V$ in $TV$. Therefore, there
exists a vector bundle $U = V^{\perp} \oplus TV^{\perp}$ on $X$ such that $V\oplus U\, =\,X \times
\mathbb{C}^{n+2r}\,=\, Q$.

We may endow $U$ with some arbitrary connection $\nabla _U$. This induces the connection
$\nabla_Q \,=\, \nabla \oplus \nabla _U$ on the bundle $Q$. Using a tubular neighborhood and a
partition of unity we may extend $\nabla_Q$ from $X$ to a connection $\nabla_{\widetilde{Q}}$
on the trivial vector bundle of rank $(n+2r)$
defined on all of $\mathbb{R}^{2n+4r}$. Now we may use Lemma \ref{trivsublemma} to come up with
a vector bundle $\widetilde{T}$ of rank $(4n+8r+2)(n+2r)$
on $\mathbb{R}^{2n+4r}$, equipped with a flat connection $\nabla _{\widetilde{T}}$, such that
$\nabla_{\widetilde{Q}}$ is induced from it. Restricting our attention to $X$ we see that
the vector bundle $T\,=\,\widetilde{T}\vert_X$ equipped with the connection
$\nabla _{\widetilde{T}} \vert_X$ satisfies the conditions in the proposition.
\end{proof}

Proposition \ref{trivsub} maybe viewed as a vector bundle version of the Nash
embedding theorem because it states that every connection arises out of a flat
connection with trivial monodromy. We now state a useful lemma \cite[Lemma 1.16]{SS}.

\begin{lemma}[Simons-Sullivan]
Let $V$ and $W$ be smooth vector bundles on a smooth manifold $X$. Let $\nabla$ be a
smooth connection on the direct sum $V\oplus W$ with curvature $R$. Let $\nabla _V$ and
$\nabla _W$ be the connections on $V$ and $W$ respectively constructed from $\nabla$ using
the decomposition of $V\oplus W$. Suppose that $R _{r,s} (V) \,\subseteq\, V$ and $R_{r,s}
(W)\,\subseteq\, W$ for all tangent vectors $r,s$ at any point of $X$. Then $$CS(\nabla _V
\oplus \nabla _W , \nabla ) \,=\, 0 \ \ \ \mathrm{modulo} \ \ \mathrm{exact}\ \
\mathrm{forms}\ \ \mathrm{on}\ \ X\, .$$
\label{sul}
\end{lemma}

Lemma \ref{sul} in conjunction with lemma \ref{trivsub} implies Theorem \ref{main} in the case $G\,=\,
{\rm GL}(N, \mathbb{C})$. Indeed, given a structured bundle $\cV \,=\,
[V, \{ \nabla _V \}$, lemma \ref{trivsub} furnishes a flat
bundle $\mathcal{T}\,=\, [T, \{ \nabla _T \}]$ such that $V$ is a subbundle of $T$ with
the connection $\nabla _V$ being induced from $\nabla _T$.
Using the natural metric on $T$ we may find an orthogonal complement $W$ to $V$ so that
$V \oplus W \,=\, T$. Endowing $W$ with the
induced connection $\nabla _W$ from $\nabla _T$ (which is flat), it is straight-forward to check
that the conditions of lemma \ref{sul} are satisfied. Let
$$\cW \,=\, [W\, , \{ \nabla _W \}]\, .$$ Using lemma \ref{sul} we see that
$$[T, \{ \nabla_{T} \}] \,=\, [V\oplus W, \{ \nabla_{T} \}] \,=\, [V\oplus
W, \{ \nabla _V \oplus \nabla _W \}] \,=\, \cV \oplus \cW\, .$$

\subsection{$\mathbf{G={\mathrm {Sp}}(2N,\mathbb{C})}$ or $\mathbf{G={\mathrm {SO}}(N,\mathbb{C})}$}
\textbf{}\label{se3.2}\\

{}From Section \ref{s2.1}
we know that there is a structured inverse $\mathcal{W} \,=\, [W\, , \{ \widetilde{\nabla} \} ]$
of $(E \, ,\nabla)$. We clarify that $W$ does not necessarily have a $G$--structure. Let
$\widetilde{\nabla}'$ denote the connection on $W^*$ induced by $\widetilde{\nabla}$.
Using the natural pairing of $W$ with $W^*$, the vector bundle $W\oplus W^*$ has a
canonical $G$--structure $\varphi_0$. We note that $\widetilde{\nabla}\oplus\widetilde{\nabla}'$ is a
$G$--connection on $(W\oplus W^*\, ,\varphi_0)$.

Clearly, $(W^*\, ,\widetilde{\nabla}')$ is a structured inverse of $(E^*\, , \nabla')$.
Therefore,
$$
(E^* \oplus W\oplus W^*\, , \nabla'\oplus \widetilde{\nabla}\oplus\widetilde{\nabla}')
$$
is a structured inverse of $(E \, ,\nabla)$. The connection
$\nabla'\oplus \widetilde{\nabla}\oplus\widetilde{\nabla}'$ preserves the $G$--structure
$\varphi'\oplus \varphi_0$ on the vector bundle $E^* \oplus W\oplus W^*$.

\section{Applications}

In this section we prove Corollary \ref{group} and Proposition \ref{holovenice}.

\subsection{Proof of Corollary \ref{group}}
\begin{enumerate}

\item Any element of $\widehat{K^0}_G(X)$ is of the form $[\cV] - [\cW]$ by definition. Since
there exists an inverse $\mathcal{Q}$ to $\cW$ such that $\mathcal{Q} \oplus \cW\,=\, [k]$,
where $[k]$ is flat with trivial monodromy, we see that $[\cV] - [\cW]\,=\, [\cV \oplus
\mathcal{Q}] - [k]$.

\item If $[\cV] \,=\, [\cW]$ in $\widehat{K^0}_G(X)$, then $\cV \oplus \mathcal{P} \,=
\,\cW \oplus \mathcal{P}$ for some structured
bundle $\mathcal{P}$. Let $\mathcal{E}$ be an inverse of $\mathcal{P}$. Adding $\mathcal{E}$
to both sides we see that $\cV
\oplus [N ] \,=\, \cW \oplus [N]$ for some flat vector bundle $N$ with trivial monodromy.
\end{enumerate}

\subsection{Proof of Proposition \ref{holovenice}}

Using the Venice lemma in \cite{PT} we see that there exists a trivial bundle
$\widetilde{T}$ with a connection $\nabla_{\widetilde{T}} \,=\, d+A$ such that
$$\frac{d\eta}{2} \,=\,
\mathrm{ch}(\widetilde{T},\nabla _{\widetilde{T}}) - \mathrm{ch} (\widetilde{T},d)\, .$$ It is
not necessarily the case that the holonomy of $\nabla_{\widetilde{T}}$ lies in $G$. However, we
know that $$\frac{d\eta}{2} \,=\,\mathrm{ch}(\widetilde{T}^{*},\nabla _{\widetilde{T}^{*}})
- \mathrm{ch} (\widetilde{T}^{*},d)$$ because $d\eta$ is an even form. Therefore,
$$d\eta \,=\,\mathrm{ch}(\widetilde{T}^{*}\oplus \widetilde{T},\nabla _{\widetilde{T}^{*}}
\oplus \nabla _{\widetilde{T}}) - \mathrm{ch}
(\widetilde{T^{*}}\oplus \widetilde{T},d)\, .$$ Using the same reasoning as in Section
\ref{se3.2} we obtain the desired result.

\section*{Acknowledgements}

We are grateful to the referee for detailed comments to improve the exposition.
The first--named author acknowledges the support of a J. C. Bose Fellowship.

\end{document}